%% file: continuous.tex
\documentclass[11pt, a4paper, notitlepage, onecolumn, oneside, openany]{amsart}

\input{packages}
\input{common_math}
\input{theorem_names_en}

\input{mathdefs}

\begin{document}
\title{A continuous time stochastic model for biological neural nets}

\author[Coregliano]{Leonardo Nagami Coregliano}

\thanks{Instituto de Matem\'atica e Estat\'istica, Universidade de S\~ao Paulo,
  \nolinkurl{lenacore@ime.usp.br}. Supported by Funda\c c\~ao de Amparo \`a Pesquisa
  do Estado de S\~ao Paulo (FAPESP) under grant no.~2013/23720-9.}

\date{\today}

\begin{abstract}
  We propose a new stochastic model for biological neural nets which is a continuous
  time version of the model proposed by~Galves and~L\"{o}cherbach
  in~\cite{MR3055382}. We also show how to computationally simulate such model for
  easy neuron potential decays and probability functions and characterize when the
  model has a finite time of death almost surely.
\end{abstract}

\maketitle

\section{Introduction}
Based on a stochastic model for biological neural nets proposed by Galves
and~L\"{o}cherbach in~\cite{MR3055382} and on a specific algorithm for simulating
such model, we propose a model on continuous time. Formally, the model is actually
still discrete but generates times of neuron discharges which could be interpreted as
a boolean stochastic process on continuous time.

One common approach to define continuous time models is to consider stochastic
differential equations, but the approach considered here is more of an algorithmic
oriented approach and throughout this work we focused on producing a model that can
be efficiently simulated.

Since the nature of the model is algorithmic, we choose to first present the model
informally with a more intuitive and constructive approach and only formalize it on
the final section.

Also in the final section we present a characterization of when the system dies
almost surely, i.e., when there is a time~$t_0\in\RR_+$ such that there are no neuron
discharges after~$t_0$.

\section{Discrete time model}
In this section we present the discrete time model proposed by~Galves
and~L\"{o}cherbach~\cite{MR3055382} and algorithms to simulate it. Such algorithms
will provide a useful insight on how to define an analogous continuous time model.

For simplicity, we will consider in this section the model with a finite set of
neurons and with no potential decay, but we stress that~Galves and~L\"{o}cherbach
did study the case of a countable set of neurons and with potential decays much more
general than the ones we will consider in the continuous time model.

\subsection{The model}
We first assume fixed a finite set~$I$ of neurons, a family of potential probability
functions~$(\phi_i)_{i\in I}$. For every~$i\in I$, the function~$\phi_i\:\RR\to[0,1]$
maps the potential of neuron~$i$ to the probability that it will fire at that
time. For every~$i,j\in I$, we also assume fixed~$W_{i\to j}\in\RR$, which gives the
influence that a discharge of neuron~$i$ has on neuron~$j$.

The model consists of a stochastic chain~$(X_t)_{t\in\ZZ}$ taking values
in~$\{0,1\}^I$ and an auxiliary chain~$(U_t)_{t\in\ZZ}$ taking values in~$\RR^I$. For
each neuron~$i\in I$, and each time~$t\in\ZZ$, the value~$X_t(i)$ will be~$1$ if
neuron~$i$ fires at time~$t$ and~$0$ otherwise. Furthermore, each neuron~$i\in I$
will have an internal potential at time~$t\in\ZZ$ given by~$U_t(i)$.

We define now the filtration
\begin{align*}
  \cF_t & = \sigma[\{X_s : s\in\ZZ, s\leq t\}], \qquad t\in\ZZ,
\end{align*}
and, for every time~$t\in\ZZ$ and neuron~$i\in I$, we let
\begin{align*}
  L_t^i & = \sup\{s \leq t : X_s(i) = 1\}
\end{align*}
be the~$F_t$-measurable time of the last discharge of neuron~$i$ up to time
time~$t$. The potential of neuron~$i$ at time~$t$ is then defined as
the~$F_t$-measurable random variable
\begin{align*}
  U_t(i) & = \max\left\{\sum_{s = L_t^i + 1}^t \sum_{j\in J}W_{j\to i}X_s(j),\; 0\right\}.
\end{align*}

The dynamics of the process is then the following: at time~$t+1$, the probability of
neuron discharges are independent conditionally on the whole past, i.e., we have
\begin{align*}
  \PP(\forall i\in I,X_{t+1}(i) = a_i\cond\cF_t)
  & = \prod_{i\in I}\PP(X_{t+1}(i) = a_i\cond\cF_t),
\end{align*}
furthermore, the probability of having a discharge from neuron~$i\in I$ at time~$t+1$
is given by
\begin{align*}
  \PP(X_{t+1}(i) = 1\cond\cF_t) & = \phi_i(U_t(i)).
\end{align*}

Note that although we define the chain depending on a great part of its past, the
probabilities of discharge at time~$t+1$ can be calculated knowing only the
potentials at time~$t$ and potentials at time~$t+1$ can be calculated knowing only
the potentials at time~$t$ and the discharges at time~$t$.

\subsection{Algorithms}
We now present two different simple algorithms to simulate the discrete model from an
initial state of potentials~$U_0$.

\subsubsection{Single-step algorithm}
The single-step algorithm consists of simulating the evolution of the process
step-by-step. Given potentials~$U_t$ at time~$t$, we sample~$|I|$ independent
uniform random variables~$(V_t(i))_{i\in I}$ over~$[0,1)$ and let
\begin{align*}
  X_{t+1}(i) & = \One_{\{V_t(i) < \phi_i(U_t(i))\}}.
\end{align*}

Furthermore, we update the potentials by letting
\begin{align*}
  U_{t+1}(i) & =
  \begin{dcases*}
    0, & if~$X_{t+1}(i) = 1$;\\
    \max\left\{U_t(i) + \sum_{j\in I}W_{j\to i}X_{t+1}(j),\; 0\right\}, &
    if~$X_{t+1}(i) = 0$;
  \end{dcases*}
  \\
  & = (1-X_{t+1}(i))\max\left\{U_t(i) + \sum_{j\in I}W_{j\to i}X_{t+1}(j),\; 0\right\}
\end{align*}

The simulation consists of repeating this process inductively from~$t=0$.

Is easy to see that this simulation indeed produces the correct stochastic process.

\subsubsection{Multi-step algorithm}
The multi-step algorithm consists of skipping steps of the single-step algorithm
where no neuron discharges occur. To do this, we first observe that in the absecence
of other neuron discharges, the time we have to wait from time~$t_0$ to see a
discharge from neuron~$i$ has geometric distribution of
parameter~$\phi_i(U_{t_0}(i))$.

Given potentials~$U_t$ at time~$t$, we sample~$|I|$ independent random
variables~$(T_t(i))_{i\in I}$ with~$T_t(i)$ having geometric distribution with
parameter~$\phi_i(U_t(i))$ for every~$i\in I$ and let~$T_t = \min\{T_t(i) : i\in I\}$
and~$D_t = \{i\in I : T_t(i) = T_t\}$.

The variable~$T_t$ gives how much time we have to wait to see a neuron firing, so we
skip this amount of time by letting~$U_{t+s}(i) = U_t(i)$ and~$X_{t+s}(i) = 0$ for
every~$0 < s < T_t$ and~$i\in I$ and then let
\begin{align*}
  X_{t+T_t}(i) & = \One_{\{i\in D_t\}}.
\end{align*}

Furthermore, we update the potentials by letting
\begin{align*}
  U_{t+T_t}(i) & = 
  \begin{dcases*}
    0, & if~$X_{t+T_t}(i) = 1$;\\
    \max\left\{U_t(i) + \sum_{j\in I}W_{j\to i}X_{t+T_t}(j),\; 0\right\},
    & if~$X_{t+T_t}(i) = 0$;
  \end{dcases*}
  \\
  & = (1-X_{t+T_t}(i))\max\left\{U_t(i) + \sum_{j\in I}W_{j\to i}X_{t+T_t}(j),\; 0\right\}.
\end{align*}

\begin{proposition}
  The multi-step algorithm correcly simulates the discrete time stochastic process.
\end{proposition}

\begin{proof}
  Throughout this proof, we denote the random variables produced by the single-step
  algorithm by~$X'$ and~$U'$.

  For every time~$t$, let
  \begin{align*}
    T_t'(i) & = \inf\{s > t : X_s'(i)=1\};\\
    T_t' & = \min\{T_t'(i) : i\in I\};\\
    D_t' & = \{i\in I : X_{T_t'}(i) = 1\}.
  \end{align*}

  Note that it is enough to prove that the random variables~$(T_0',D_0')$
  and~$(T_0,D_0)$ have the same law (because after times~$T_0$ and~$T_0'$, the
  algorithms depend only on~$U_{T_0}$ and~$U_{T_0'}'$ respectively).

  Observe now that for every~$t_0>0$ and every~$A_0\subset I$ not-empty, we have
  that~$(T_0',D_0') = (t_0,A_0)$ if and only if
  \begin{align*}
    X_t'(i) & = 0, \qquad \forall i\in I, \forall t < t_0;\\
    X_{t_0}'(i) & = 0, \qquad \forall i\in I\setminus A_0;\\
    X_{t_0}'(i) & = 1, \qquad \forall i\in A_0;
  \end{align*}
  hence we have
  \begin{align*}
    \PP((T_0',D_0') = (t_0,A_0))
    & = \prod_{t=1}^{t_0-1}\prod_{i\in I}(1-\phi_i(U_0(i)))
    \prod_{i\in I\setminus A_0}(1-\phi_i(U_0(i)))
    \prod_{i\in A_0}\phi_i(U_0(i))
    \\
    & = \prod_{i\in I\setminus A_0}(1-\phi_i(U_0(i)))^{t_0}
    \prod_{i\in A_0}(1-\phi_i(U_0(i)))^{t_0-1}\phi_i(U_0(i)),
  \end{align*}
  because the potentials of the neurons remain the same if there is no neuron discharge.

  Therefore we have
  \begin{align*}
    \PP((T_0',D_0') = (t_0,A_0))
    & = \prod_{i\in I\setminus A_0}\PP(T_0(i) > t_0)
    \prod_{i\in A_0}\PP(T_0(i) = t_0)
    \\
    & = \PP(\forall i\in I\setminus A_0, T_0(i) > t_0\text{ and }\forall i\in A_0,
    T_0(i) = t_0)
    \\
    & = \PP((T_0,D_0) = (t_0,A_0),
  \end{align*}
  because of the independence of the random variables~$(T_0(i))_{i\in I}$.

  Thus we have proved that~$(T_0',D_0')$ and~$(T_0,D_0)$ have the same law.
\end{proof}

\section{Continuous time model}
\label{sec:continuous}
We introduce now the continuous time stochastic process. The idea is to define the
process from the multi-step algorithm presented in the previous section.

We suppose that an initial potential~$(U_0(i))_{i\in I}$ is given. Naturally, the
difference from this model to the last is that the chains~$(X_t)_{t\in\RR_+}$
and~$(U_t)_{t\in\RR_+}$ are now indexed by continuous time.

Now we would like to add decay on the potentials of the neurons so we suppose that,
for every~$i\in I$, we are given a
function~$V_i\:\RR_+\times\RR_+\to\RR_+$ that will give how the
potential of neuron~$i$ decays in the abscence of neuron discharges: the
value~$V_i(u,s)$ will be how much potential is left after a total time~$s$ has
passed if the neuron started with potential~$u$ (naturally, we suppose~$V_i(u,0) = u$
for every~$u>0$).

The dynamics of the process is then defined by the following evolution algorithm.

Given potentials~$U_t$ at time~$t$, we sample~$|I|$ independent random
variables~$(T_t(i))_{i\in I}$ with~$T_t(i)$ having a certain distribution (which will
depend on the choice of~$V_i$ and which we will study later) for every~$i\in I$ and
let~$T_t = \min\{T_t(i) : i\in I\}$ and~$D_t = \{i\in I : T_t(i) = T_t\}$.

Then we let~$U_{t+s}(i) = V_i(U_t(i),s)$ and~$X_{t+s}(i) = 0$ for
every~$0 < s < T_t$ and~$i\in I$ and let
\begin{align*}
  X_{t+T_t}(i) & = \One_{\{i\in D_t\}}.
\end{align*}

Furthermore, we update the potentials by letting
\begin{align*}
  U_{t+T_t}(i) & = 
  \begin{dcases*}
    0, & if~$X_{t+T_t}(i) = 1$;\\
    \max\left\{V_i(U_t(i),T_t) + \sum_{j\in I}W_{j\to i}X_{t+T_t}(j),\; 0\right\},
    & if~$X_{t+T_t}(i) = 0$;
  \end{dcases*}
  \\
  & = (1-X_{t+T_t}(i))
  \max\left\{V_i(U_t(i),T_t) + \sum_{j\in I}W_{j\to i}X_{t+T_t}(j),\; 0\right\}.
\end{align*}

\subsection{Distribution of wait time}
In this section, we will study what kind of distribution we should put in the random
variables~$T_t(i)$'s. We assume that~$t$ and~$i$ are fixed, and we will
denote~$T_t(i)$ by simply~$T$ until the end of this section
(Section~\ref{sec:continuous}).

In the multi-step algorithm of the discrete time model, this variable had geometric
distribution with parameter~$\lambda = \phi(U)$ (here we also are dropping the
notation by letting~$\phi\equiv\phi_i$ and~$U = U_t(i)$). This means that, for
every~$k\in\NN^*$, we had
\begin{align*}
  \PP(T = k) & = (1-\lambda)^{k-1}\lambda = \PP(T < k)\lambda,
\end{align*}
so the natural analogous distribution in continuous time is to consider a probability
density function~$\rho$ satisfying
\begin{align*}
  \rho(t) & = \(1 - \int_0^t\rho(s)ds\)\lambda.
\end{align*}

Now we have that~$\rho$ must be a solution of the following differential equation
\begin{align*}
  \rho'(t) & = -\lambda\rho(t),
\end{align*}
hence it is of the form
\begin{align*}
  \rho(t) & = Ce^{-\lambda t},
\end{align*}
where~$C\in\RR$ is a constant.

Substituting back in the original equation, we have
\begin{align*}
  \int_0^t\rho(s)ds & = 1 - \frac{\rho(t)}{\lambda} = 1 - \frac{Ce^{-\lambda t}}{\lambda},
\end{align*}
and, using the condition~$\int_0^0\rho(s)ds = 0$, we deduce that~$C = \lambda$,
hence we obtain the exponential distribution of parameter~$\lambda$, which was
expected.

But we are now interested in allowing the parameter~$\lambda$ to vary in time. This
means that~$\rho$ should satisfy
\begin{align*}
  \rho(t) & = \(1 - \int_0^t\rho(s)ds\)\lambda(t),
\end{align*}
which leads to the following differential equation
\begin{align*}
  \frac{\rho'(t)\lambda(t) - \rho(t)\lambda'(t)}{\lambda^2(t)} & = \rho(t),
\end{align*}
whose solutions are of the form
\begin{align*}
  \rho(t) & = C\lambda(t)\exp\(-\int_0^t\lambda(s)ds\),
\end{align*}
where~$C\in\RR$ is a constant.

Substituting back in the original equation, we have
\begin{align*}
  \int_0^t\rho(s)ds & = 1 - \frac{\rho(t)}{\lambda(t)}
  = 1 - C\exp\(-\int_0^t\lambda(s)ds\),
\end{align*}
and again, using the condition~$\int_0^0\rho(s)ds$, we deduce that~$C = 1$.

Note now that~$\int_0^{+\infty}\rho(s)ds = 1$ if and only
if~$\int_0^{+\infty}\lambda(s)ds = +\infty$. The interpretation of this is that
if~$\lambda(t)$ decreases too fast, then there is a non-zero probability that the
neuron will never fire.

Now define the function
\begin{align*}
  \begin{functiondef}
    F\: & [0,+\infty] & \longrightarrow & [0,1]\\
    & t & \longmapsto &
    \begin{dcases*}
      \int_0^t\rho(s)ds, & if~$t<+\infty$;\\
      1, & if~$t=+\infty$;
    \end{dcases*}
  \end{functiondef}
\end{align*}
which is the analogous of a cumulative distribution function with the difference that
the random variable can take the value~$+\infty$.

Therefore, the distribution of~$T$ is given by
\begin{align*}
  \begin{dcases}
    \PP(T < t) = F(t), & \forall t < +\infty;\\
    \PP(T = +\infty) = F(+\infty)-\lim_{s\to+\infty}F(s)
    = 1 - \int_0^{+\infty}\rho(s)ds. &
  \end{dcases}
\end{align*}

Now we let~$\lambda(t) = \phi(V(t))$ (again dropping the notation by letting~$V(s) =
V_i(U(t),s)$).

In the next sections, we will study decay laws that satisfy the following
differential equation
\begin{align*}
  V'(t) & = -\mu V^\gamma(t),
\end{align*}
where~$\gamma\geq 1$ is a fixed constant.

Note that in this particular case we have
\begin{align*}
  \int_0^t\lambda(s)ds & = \int_0^t\phi(V(s))ds 
  = \int_0^t\frac{\phi(V(s))V'(s)}{-\mu V^\gamma(s)}ds
  = \frac{1}{\mu}\int_{V(t)}^{V(0)}\frac{\phi(v)}{v^\gamma}dv.
\end{align*}
and~$\lim_{t\to+\infty}V(t) = 0$.

Furthermore, we are interested in the following potential probability functions
\begin{enumerate}
\item Exponential: $\phi^{(\exp)}(u) = 1 - e^{-\beta u}$, for~$\beta > 0$ fixed;
\item Rational: $\phi^{(r)}(u) = u^r / (u^r + \beta)$, for~$r\in\NN^*$ and~$\beta >
  0$ fixed;
\item Monomial: $\phi_r(u) = \beta u^r$, for~$r\in\NN^*$ and~$\beta > 0$ fixed;
\end{enumerate}

\subsection{Exponential decay}
In the case~$\gamma = 1$, we have the following differential equation
\begin{align*}
  V'(t) & = -\mu V(t),
\end{align*}
whose solutions are of the form
\begin{align*}
  V(t) & = V(0)e^{-\mu t},
\end{align*}
which is the exponential decay typical of radioactive decay models.

For this type of decay, we can prove a very interesting property.

\begin{proposition}\label{prop:pinftyexp}
  If~$g\:\RR_+\to\RR_+$ is a continuous function with~$g(u) = 0$ if and only if~$u =
  0$ and continuously differentiable in a neighbourhood of~$0$ and~$r > 0$ is a
  positive real number, then taking~$V(t) = V(0)e^{-\mu t}$ and~$\phi\equiv g^r$
  yields~$\PP(T = +\infty)>0$.
\end{proposition}

\begin{proof}
  Since~$g$ is continuously differentiable in a neighbourhood of~$0$ and~$g(0) = 0$,
  for every~$\epsilon > 0$, there exists~$t_0>0$ such that
  \begin{align*}
    g(t) & \leq (g'(0)+\epsilon)t,
  \end{align*}
  for every~$0 \leq t < t_0$, hence
  \begin{align*}
    \phi(t) & \leq (g'(0)+\epsilon)^rt^r.
  \end{align*}

  This means that
  \begin{align*}
    \int_0^{t_0}\frac{\phi(v)}{v}dv & \leq \int_0^{t_0}\frac{(g'(0)+\epsilon)^rv^r}{v}dv
    \\
    & = (g'(0)+\epsilon)^r\left.\frac{t^r}{r}\right\vert_0^{t_0}
    = (g'(0)+\epsilon)^r\frac{t_0^r}{r} < +\infty,
  \end{align*}
  where the last two equalities follow from the fact that~$r > 0$.

  Hence
  \begin{align*}
    \int_0^{+\infty}\lambda(s)ds & = \frac{1}{\mu}\int_0^{V(0)}\frac{\phi(v)}{v}dv
    \\
    & = \frac{1}{\mu}\(\int_0^{t_0}\frac{\phi(v)}{v}dv
    + \int_{t_0}^{V(0)}\frac{\phi(v)}{v}dv\)
    \\
    & < +\infty
  \end{align*}
  and, therefore
  \begin{align*}
    \PP(T < +\infty) & = 1 - \exp\(-\int_0^{+\infty}\lambda(s)ds\) < 1.
  \end{align*}
\end{proof}

\begin{corollary}
  Suppose that, for every~$i\in I$, we are given a positive real number~$r_i>0$ and a
  function~$g_i\:\RR_+\to\RR_+$ such that~$g_i(0) = 0$ and such that~$g_i$ is
  continuous and continuously differentiable in a neighbourhood of~$0$.

  Then the continuous time model with~$\phi_i\equiv g_i^r$ for every~$i\in I$ and
  exponential decay has non-null probability of never firing, i.e., we have
  \begin{align*}
    \PP(\forall t > 0, \forall i\in I, X_t(i) = 0) > 0.
  \end{align*}
\end{corollary}

\begin{remark*}
  Clearly from the proof of Proposition~\ref{prop:pinftyexp}, we only need Lipschitz
  condition on a neighbourhood of~$0$ and locally boundedness of~$g$, but we chose
  not to state the proposition in the more general form for simplicity.
\end{remark*}

\subsubsection{Pratical sampling}
Although the continuous time algorithm is well-defined, we are interested in actually
implementing it with a reasonable efficiency, so in this section we show how one can
sample the random variable~$T$ of the algorithm from a uniform random variable~$Z$
over~$[0,1)$.

For the case of the rational potential~$\phi^{(r)}$, note that, for every~$t<+\infty$,
we have
\begin{align*}
  \int_0^t\lambda(s)ds & = \frac{1}{\mu}\int_{V(t)}^{V(0)}\frac{\phi^{(r)}(v)}{v}dv
  \\
  & = \frac{1}{\mu}\int_{V(t)}^{V(0)}\frac{v^{r-1}}{v^r+\beta}dv
  \\
  & = \frac{1}{\mu}\left.\frac{\ln(v^r+\beta)}{r}\rv_{V(t)}^{V(0)}
  \\
  & = \frac{1}{r\mu}\ln\frac{V(0)^r+\beta}{V(t)^r+\beta}
  \tendsto^{t\to+\infty}\frac{1}{r\mu}\ln\frac{V(0)^r+\beta}{\beta}.
\end{align*}

Hence, for every~$t<+\infty$, we have
\begin{align*}
  F(t) & = 1 - \exp\(-\int_0^t\lambda(s)ds\)
  = 1 - \(\frac{V(0)^r+\beta}{V(t)^r+\beta}\)^{-1/(r\mu)}
  \\
  & = 1 - \(\frac{V(0)^r+\beta}{V(0)^re^{-\mu rt}+\beta}\)^{-1/(r\mu)}
  \tendsto^{t\to+\infty} 1 - \(\frac{V(0)^r+\beta}{\beta}\)^{-1/(r\mu)}.
\end{align*}

With a simple calculation, for every~$t<+\infty$, we have
\begin{align*}
  t & = -\frac{1}{r\mu}\ln\frac{(V(0)^r+\beta)(1 - F(t))^{r\mu}-\beta}{V(0)^r},
\end{align*}
hence defining
\begin{align*}
  \begin{functiondef}
    G\: & [0,1) & \longrightarrow & [0,+\infty]\\
    & \xi & \longmapsto &
    \begin{dcases}
      -\frac{1}{r\mu}\ln\frac{(V(0)^r+\beta)(1 - \xi)^{r\mu}-\beta}{V(0)^r},
      & \text{ if }\xi < 1 - \(\frac{V(0)^r+\beta}{\beta}\)^{-1/(r\mu)};
      \\
      +\infty, & \text{ otherwise};
    \end{dcases}
  \end{functiondef}
\end{align*}
we have~$\PP(G(Z) < t) = F(t)$ for every~$t<+\infty$
and~$\PP(G(Z)<+\infty)=\lim_{t\to+\infty}F(t)$, thus giving an easy way of
sampling~$T$ from~$Z$.

\begin{remark*}
  Note that taking the limit~$V(0)\to+\infty$, we have that~$G(\xi) = -\ln(1-\xi)$
  for every~$\xi\in[0,1)$, this means that if~$V(0)$ is arbitrarily large, the wait
  time has exponential distribution with parameter~$1$ ($\xi = 1 - e^{-G(\xi)}$).
\end{remark*}

\bigskip

The case of the exponential potential~$\phi^{(\exp)}$ also gives an
invertible~$F(t)$, but the inverse cannot be expressed in terms of essential
functions, which makes it very unpratical to compute.

\bigskip

We will now consider the monomial potential~$\phi_r$. Note that, for
every~$t<+\infty$, we have
\begin{align*}
  \int_0^t\lambda(s)ds & = \frac{1}{\mu}\int_{V(t)}^{V(0)}\frac{\phi_r(v)}{v}dv
  \\
  & = \frac{1}{\mu}\int_{V(t)}^{V(0)}\beta v^{r-1}dv
  \\
  & = \frac{\beta (V(0)^r - V(t)^r)}{r\mu}
  \\
  & \tendsto^{t\to+\infty}\frac{\beta V(0)^r}{r\mu}
\end{align*}

Hence, for every~$t<+\infty$, we have
\begin{align*}
  F(t) & = 1 - \exp\(-\int_0^t\lambda(s)ds\)
  = 1 - \exp\(-\frac{\beta(V(0)^r - V(t)^r)}{r\mu}\)
  \\
  & = 1 - \exp\(-\frac{\beta V(0)^r(1 - e^{-\mu rt})}{r\mu}\)
  \tendsto^{t\to+\infty} 1 - \exp\(-\frac{\beta V(0)^r}{r\mu}\).
\end{align*}

With a simple calculation, for every~$t<+\infty$, we have
\begin{align*}
  t & = -\frac{1}{\mu r}\ln\(1-\frac{r\mu\ln(1-F(t))}{\beta V(0)^r}\),
\end{align*}
hence defining
\begin{align*}
  \begin{functiondef}
    G\: & [0,1) & \longrightarrow & [0,+\infty]\\
    & \xi & \longmapsto &
    \begin{dcases}
      -\frac{1}{\mu r}\ln\(1-\frac{r\mu\ln(1-\xi)}{\beta V(0)^r}\),
      & \text{ if }\xi < 1 - \exp\(-\frac{\beta V(0)^r}{r\mu}\);
      \\
      +\infty, & \text{ otherwise};
    \end{dcases}
  \end{functiondef}
\end{align*}
we have~$\PP(G(Z) < t) = F(t)$ for every~$t<+\infty$
and~$\PP(G(Z)<+\infty)=\lim_{t\to+\infty}F(t)$, thus giving an easy way of
sampling~$T$ from~$Z$.

\begin{remark*}
  Note that taking the limit~$V(0)\to+\infty$, we have that~$G(\xi) = 0$ for
  every~$\xi\in[0,1)$, this means that if~$V(0)$ is arbitrarily large, the neuron
  fires imediately. This behaviour is typical of potential functions such
  that~$\lim_{u\to+\infty}\phi(u)=+\infty$.
\end{remark*}

\subsection{Other decays}
Now we will cover the case where~$V(t)$ satisfies the
following differential equation
\begin{align*}
  V'(t) & = -\mu V(t)^\gamma,
\end{align*}
with~$\gamma > 1$ a fixed constant.

The solutions of this differential equation are of the form
\begin{align*}
  V(t) & = (V(0)^{1-\gamma} + (\gamma-1)\mu t)^{1/(1-\gamma)},
\end{align*}
which in the particular case when~$\gamma=2$ gives the reciprocal decay
\begin{align*}
  V(t) & = \frac{1}{\mu t + V(0)^{-1}}.
\end{align*}

And in this cases we have a proposition analogous to
Proposition~\ref{prop:pinftyexp}.

\begin{proposition}\label{prop:pinftyother}
  Suppose~$g\:\RR_+\to\RR_+$ is a continuous function vanishing only at~$0$ and of
  class~$\cC^{\floor{\gamma}}$ on a neighbourhood of~$0$ and
  \begin{align*}
    \frac{d^kg}{dt^k}(0) & = 0,
  \end{align*}
  for every~$k = 0, 1, \ldots, \floor{\gamma}-1$.

  Then for every~$r > (\gamma-1)/\floor{\gamma}$, taking~$V(t) = (V(0)^{1-\gamma} +
  (\gamma-1)\mu t)^{1/(1-\gamma)}$ and~$\phi\equiv g^r$ yields~$\PP(T = +\infty)>0$.
\end{proposition}

\begin{proof}
  Note that~$\floor{\gamma}r - \gamma + 1> 0$. Since~$g$ is of
  class~$\cC^{\floor{\gamma}}$ on a neighbourhood of~$0$ and all
  the~$\floor{\gamma}-1$ first derivatives of~$g$ are null on~$0$, we have that for
  every~$\epsilon>0$, there exists~$t_0>0$ such that
  \begin{align*}
    g(t) & \leq (g^{(\floor{\gamma})}(0) + \epsilon)\frac{t^{\floor{\gamma}}}{\floor{\gamma}!},
  \end{align*}
  for every~$0<t<t_0$, where~$g^{(\floor{\gamma})}$ is the~$\floor{\gamma}$-th derivative
  of~$g$. Hence we have
  \begin{align*}
    \phi(t) & \leq
    (g^{(\floor{\gamma})}(0) + \epsilon)^r\frac{t^{\floor{\gamma}r}}{\floor{\gamma}!^r}.
  \end{align*}

  This means that
  \begin{align*}
    \int_0^{t_0}\frac{\phi(v)}{v^\gamma}dv & \leq
    \int_0^{t_0}\frac{(g^{(\floor{\gamma})}(0) + \epsilon)^rt^{\floor{\gamma}r}}%
        {\floor{\gamma}!^rv^\gamma}dv
    \\
    & = \(\frac{g^{(\floor{\gamma})}(0)+\epsilon}{\floor{\gamma}!}\)^r
    \left.\frac{t^{\floor{\gamma}r - \gamma + 1}}{\floor{\gamma}r - \gamma + 1}
    \right\vert_0^{t_0}
    \\
    & = \(\frac{g^{(\floor{\gamma})}(0)+\epsilon}{\floor{\gamma}!}\)^r
    \frac{t_0^{\floor{\gamma}r}}{\floor{\gamma}r},
  \end{align*}
  where the last two equalities follow from the fact
  that~$\floor{\gamma}r-\gamma+1>0$.

  Hence
  \begin{align*}
    \int_0^{+\infty}\lambda(s)ds & = \frac{1}{\mu}\int_0^{V(0)}\frac{\phi(v)}{v^\gamma}dv
    \\
    & = \frac{1}{\mu}\(\int_0^{t_0}\frac{\phi(v)}{v^\gamma}dv
    + \int_{t_0}^{V(0)}\frac{\phi(v)}{v^\gamma}dv\)
    \\
    & < +\infty
  \end{align*}
  and, therefore
  \begin{align*}
    \PP(T < +\infty) & = 1 - \exp\(-\int_0^{+\infty}\lambda(s)ds\) < 1.
  \end{align*}
\end{proof}

\begin{corollary}
  Suppose that, for every~$i\in I$, we are given a real number~$r_i >
  (\gamma-1)/\floor{\gamma}$ function~$g_i\:\RR_+\to\RR_+$ such that~$g_i$ is
  continuous and continuously differentiable $\floor{\gamma}$ times in a
  neighbourhood of~$0$ and the first~$\floor{\gamma}-1$ derivatives of~$g$ and~$g$
  itself are all null on~$0$.

  Then the continuous time model with~$\phi_i\equiv g_i^r$ for every~$i\in I$ and
  decay with parameter~$\gamma$ has non-null probability of never firing, i.e., we
  have
  \begin{align*}
    \PP(\forall t > 0, \forall i\in I, X_t(i) = 0) > 0.
  \end{align*}
\end{corollary}

\begin{remark*}
  Clearly from the proof of Proposition~\ref{prop:pinftyother}, we only need
  $\floor{\gamma}$-H\"{o}lder condition on a neighbourhood of~$0$ and locally
  boundedness of~$g$, but again we chose not to state the proposition in the more
  general form for simplicity.
\end{remark*}

\subsubsection{Pratical sampling}
Again we show that for some potentials we can sample the random variable~$T$ of the
algorithm from a uniform random variable~$Z$ over~$[0,1)$.

For simplicity, we will only show the cases with~$\gamma = 2$.

For the rational potential~$\phi^{(r)}$ with~$r=1$, note that, for every~$t<+\infty$,
we have
\begin{align*}
  \int_0^t\lambda(s)ds & = \frac{1}{\mu}\int_{V(t)}^{V(0)}\frac{\phi^{(1)}(v)}{v^2}dv
  \\
  & = \frac{1}{\mu}\int_{V(t)}^{V(0)}\frac{1}{v(v+\beta)}dv
  \\
  & = \frac{1}{\mu}
  \int_{V(t)}^{V(0)}\(\frac{\beta^{-1}}{v} - \frac{\beta^{-1}}{v+\beta}\)dv
  \\
  & = \frac{1}{\beta\mu}\left.\ln\frac{v}{v+\beta}\rv_{V(t)}^{V(0)}
  \\
  & = \frac{1}{\beta\mu}\ln\frac{V(0)(V(t)+\beta)}{V(t)(V(0)+\beta)}
  \tendsto^{t\to+\infty} +\infty.
\end{align*}

Hence, for every~$t<+\infty$, we have
\begin{align*}
  F(t) & = 1 - \exp\(-\int_0^t\lambda(s)ds\)
  = 1 - \(\frac{V(0)(V(t)+\beta)}{V(t)(V(0)+\beta)}\)^{-1/(\beta\mu)}
  \\
  & = 1 - \(\frac{V(0)((\mu t + V(0)^{-1})^{-1}+\beta)}%
    {(\mu t + V(0)^{-1})^{-1}(V(0)+\beta)}\)^{-1/(\beta\mu)}
  \tendsto^{t\to+\infty} 1.
\end{align*}

With a simple calculation, for every~$t<+\infty$, we have
\begin{align*}
  t & = \frac{(1-F(t))^{-\beta\mu} - 1}{\mu}\(\frac{1}{V(0)} + \frac{1}{\beta}\),
\end{align*}
hence defining
\begin{align*}
  \begin{functiondef}
    G\: & [0,1) & \longrightarrow & [0,+\infty]\\
    & \xi & \longmapsto & \frac{(1-\xi)^{-\beta\mu} - 1}{\mu}
      \(\frac{1}{V(0)} + \frac{1}{\beta}\)
  \end{functiondef}
\end{align*}
we have~$\PP(G(Z) < t) = F(t)$ for every~$t<+\infty$, thus giving an easy way of
sampling~$T$ from~$Z$.

\begin{remark*}
  Note that in this case~$G$ ommits the value~$+\infty$, which means that the neuron
  will always fire in finite time.

  Note also that taking the limit~$V(0)\to+\infty$, we have
  that~$G(\xi)=((1-\xi)^{-\beta\mu}-1)/(\beta\mu)$ for every~$\xi\in[0,1)$.
\end{remark*}

\bigskip

For the rational potential~$\phi^{(r)}$ with~$r=2$, note that, for every~$t<+\infty$,
we have
\begin{align*}
  \int_0^t\lambda(s)ds & = \frac{1}{\mu}\int_{V(t)}^{V(0)}\frac{\phi^{(2)}(v)}{v^2}dv
  \\
  & = \frac{1}{\mu}\int_{V(t)}^{V(0)}\frac{1}{v^2+\beta}dv
  \\
  & = \frac{1}{\mu}\left.
  \frac{1}{\sqrt{\beta}}\arctan\frac{v}{\sqrt{\beta}}\rv_{V(t)}^{V(0)}
  \\
  & = \frac{1}{\sqrt{\beta}\mu}
  \(\arctan\frac{V(0)}{\sqrt{\beta}} - \arctan\frac{V(t)}{\sqrt{\beta}}\)
  \tendsto^{t\to+\infty} \frac{1}{\sqrt{\beta}\mu}\arctan\frac{V(0)}{\sqrt{\beta}}.
\end{align*}

Hence, for every~$t<+\infty$, we have
\begin{align*}
  F(t) & = 1 - \exp\(-\int_0^t\lambda(s)ds\)
  = 1 - \exp\(-\frac{1}{\sqrt{\beta}\mu}
  \(\arctan\frac{V(0)}{\sqrt{\beta}} - \arctan\frac{V(t)}{\sqrt{\beta}}\)\)
  \\
  & = 1 - \exp\(-\frac{1}{\sqrt{\beta}\mu}
  \(\arctan\frac{V(0)}{\sqrt{\beta}}
  - \arctan\frac{(\mu t + V(0)^{-1})^{-1}}{\sqrt{\beta}}\)\)
  \\
  & \tendsto^{t\to+\infty} 1 - \exp\(-\frac{1}{\sqrt{\beta}\mu}
  \arctan\frac{V(0)}{\sqrt{\beta}}\).
\end{align*}

Again with a simple calculation, for every~$t<+\infty$, we have
\begin{align*}
  t & = \frac{1}{\mu}\(\frac{1}{\sqrt{\beta}}
  \cot\(\arctan\frac{V(0)}{\sqrt{\beta}} + \mu\sqrt{\beta}\ln(1-F(t))\)
   - \frac{1}{V(0)}\),
\end{align*}
hence defining
\begin{align*}
  \begin{functiondef}
    G\: & [0,1) & \longrightarrow & [0,+\infty]\\
    & \xi & \longmapsto &
    \begin{dcases}
      \frac{1}{\mu}\(\frac{1}{\sqrt{\beta}}
      \cot\(\arctan\frac{V(0)}{\sqrt{\beta}} + \mu\sqrt{\beta}\ln(1-\xi)\)
      - \frac{1}{V(0)}\),
      & \text{ if }\xi < L;
      \\
      +\infty, & \text{ otherwise};
    \end{dcases}
  \end{functiondef}
\end{align*}
where
\begin{align*}
  L & = 1 - \exp\(-\frac{1}{\sqrt{\beta}\mu}\arctan\frac{V(0)}{\sqrt{\beta}}\),
\end{align*}
we have~$\PP(G(Z) < t) = F(t)$ for every~$t<+\infty$ and~$\PP(G(Z) < +\infty) =
\lim_{t\to+\infty}F(t)$, thus giving an easy way of sampling~$T$ from~$Z$.

\begin{remark*}
  Note that taking the limit~$V(0)\to+\infty$, we have that
  \begin{align*}
    G(\xi) & =
    \begin{dcases}
      \frac{1}{\sqrt{\beta}\mu}\tan\(-\mu\sqrt{\beta}\ln(1-\xi)\),
      & \text{ if } \xi < 1 - \exp\(-\frac{\pi}{2\sqrt{\beta}\mu}\);
      \\
      +\infty, & \text{ otherwise};
    \end{dcases}
  \end{align*}
  for every~$\xi\in[0,1)$, this means that even if~$V(0)$ is arbitrarily large, there
  is a probability of at least~$\exp(-\pi/(2\sqrt{\beta}\mu))$ that the neuron will
  not fire.
\end{remark*}

\bigskip

Now, for the monomial decay~$\phi_r$ with~$r=1$, for every~$t<+\infty$, we have
\begin{align*}
  \int_0^t\lambda(s)ds & = \frac{1}{\mu}\int_{V(t)}^{V(0)}\frac{\phi_1(v)}{v^2}dv
  \\
  & = \frac{1}{\mu}\int_{V(t)}^{V(0)}\frac{\beta}{v}dv
  \\
  & = \frac{\beta}{\mu}\ln\frac{V(0)}{V(t)}
  \\
  & \tendsto^{t\to+\infty} +\infty.
\end{align*}

Hence, for every~$t<+\infty$, we have
\begin{align*}
  F(t) & = 1 - \exp\(-\int_0^t\lambda(s)ds\)
  = 1 - \(\frac{V(0)}{V(t)}\)^{-\beta/\mu}
  \\
  & = 1 - \(1 + V(0)\mu t\)^{-\beta/\mu}
  \tendsto^{t\to+\infty} 1.
\end{align*}

With a simple calculation, for every~$t<+\infty$, we have
\begin{align*}
  t & = \frac{(1 - F(t))^{-\mu/\beta} - 1}{\mu V(0)},
\end{align*}
hence defining
\begin{align*}
  \begin{functiondef}
    G\: & [0,1) & \longrightarrow & [0,+\infty]\\
    & \xi & \longmapsto & \frac{(1 - \xi)^{-\mu/\beta} - 1}{\mu V(0)},
  \end{functiondef}
\end{align*}
we have~$\PP(G(Z) < t) = F(t)$ for every~$t<+\infty$, thus giving an easy way of
sampling~$T$ from~$Z$.

\begin{remark*}
  Note that taking the limit~$V(0)\to+\infty$, we have that~$G(\xi) = 0$, this means
  that if~$V(0)$ is arbitrarily large, the neuron fires imediately.
\end{remark*}

\bigskip

Last, but not least, we consider the monomial potential~$\phi_r$ with~$r\geq 2$. Note
that, for every~$t<+\infty$, we have
\begin{align*}
  \int_0^t\lambda(s)ds & = \frac{1}{\mu}\int_{V(t)}^{V(0)}\frac{\phi_r(v)}{v^2}dv
  \\
  & = \frac{1}{\mu}\int_{V(t)}^{V(0)}\beta v^{r-2}dv
  \\
  & = \frac{\beta}{(r-1)\mu}(V(0)^{r-1} - V(t)^{r-1})
  \\
  & \tendsto^{t\to+\infty} +\infty,
\end{align*}
which coincides with the case~$\phi_{r-1}$ with exponential decay (with a
different~$\beta$).

\bigskip

We hope that these examples were enough to illustrate that by changing the potential
decay law and the potential probability functions we can obtain very distinct wait
time distributions.

\section{Formalization}

We will now formalize the model.

Suppose~$I$ is a finite set and~$(V_i)_{i\in I}\in\(\RR_+^{\RR_+\times\RR_+}\){}^I$
and~$(\phi_i)_{i\in I}\in\(\RR_+^{\RR_+}\){}^I$ are families of functions such that

\begin{enumerate}[1.]
\item For every~$u\in\RR_+$ and every~$i\in I$, we have~$V_i(u,0)=u$;
\item For every~$t\in\RR_+$ and every~$i\in I$, the function~$V_i({{}\cdot{}},t)$ is
  non-decreasing;
\item For every~$u\in\RR_+$ and every~$i\in I$, the function~$V_i(u,{{}\cdot{}})$ is
  non-increasing;
\item For every~$u\in\RR_+$, every~$i\in I$ and every~$t,t'\in\RR_+$, we have
  \begin{align*}
    V_i(u,t+t') & = V_i(V_i(u,t),t');
  \end{align*}
\item For every~$i\in I$, the function~$\phi_i$ is non-decreasing;
\item For every~$i\in I$, the function~$\phi_i\comp V_i$ is Borel-measurable (on both
  coordinates);
\end{enumerate}

Suppose furthermore that~$W_{i\to j}\in\RR^{I\times I}$ is a matrix of neuron
influences and that~$(Z_{i,n})_{i\in I, n\in\NN}$ is a family of independent
identically distributed uniform random variables over~$[0,1)$.

Suppose finally that~$(a(i))_{i\in I}\in\RR_+^I$ is a vector of initial potentials.

We now define sequences~$(T_n(i))_{n\in\NN, i\in I}$, $(X_n(i))_{n\in\NN, i\in I}$
and~$(U_n(i))_{n\in\NN, i\in I}$ of random variables inductively as follows.

\begin{enumerate}[i.]
\item Let~$U_0(i) = a(i)$ for every~$i\in I$;
\item For every~$n\in\NN$, let
  \begin{align*}
    T_n(i) & = \sup \left\{t \in[0,+\infty) :
    1 - \exp\(-\int_0^t\phi_i(V_i(U_n(i),s))dm(s)\) \leq Z_{i,n}\right\};
  \end{align*}

\item For every~$n\in\NN$, let~$T_n = \min\{T_n(i) : i\in I\}$;

  If~$T_n = +\infty$, let~$U_m(i) = \lim_{t\to+\infty}V_i(U_n(i),t)$,
  $T_m(i)=T_m=+\infty$ and~$X_m(i) = 0$ for every~$i\in I$ and every~$m > n$ and stop
  the induction;
\item For every~$n\in\NN$, let~$X_n(i) = \One_{\{T_n(i) = T_n\}}$;
\item For every~$n\in\NN$, let
  \begin{align*}
    U_{n+1}(i) & = (1-X_n(i))
    \max\left\{\lim_{t\to T_n^-}V_i(U_n(i),t) + \sum_{j\in I}W_{j\to i}X_n(j),\; 0\right\}.
  \end{align*}
\end{enumerate}

Finally, define, for every~$n\in\NN$, define~$T_n' = \sum_{n\in\NN}T_n$.

\begin{observation*}
  The condition
  \begin{align*}
    V_i(u,t+t') & = V_i(V_i(u,t),t')
  \end{align*}
  might seem unnatural at first, but it is an underlying condition of the algorithm
  because it says that interrupting the potential evolution at any time~$t$ and
  restarting the evolution from the value~$V_i(u,t)$ yields the same result of not
  interrupting. Such interruptions are made in the algorithm whenever a neuron fires
  and this condition says that a neuron~$i$ will not change its evolution law only
  because another completely independent neuron~$j$ fired (by completely independent,
  we mean a neuron whose influence~$W_{j\to i}$ is zero).

  Note that the values~$(W_{i\to i})_{i\in I}$ are completely irrelevant (they are
  included in the definition of the model for notational simplicity only), because
  if~$X_n(i) = 1$, we have~$U_{n+1}(i) = 0$.

  Moreover, note that formally we can't define random variables indexed by~$\RR_+$ in
  the general case because we might have~$\lim_{n\to+\infty}T_n'<+\infty$.

  Finally, note that we almost surely never have two neurons firing at the same time
  (i.e., we have~$\forall n\in\NN, \sum_{i\in I}X_n(i) \leq 1$ almost surely).
\end{observation*}

Below, we present a sufficient condition for having~$\lim_{n\to+\infty}T_n'=+\infty$.

\begin{lemma}\label{lem:Tnptoinfinity}
  If, for every~$i\in I$, there exists~$t_i>0$ such that
  \begin{align*}
    \sup_{u\in\RR_+}\int_0^{t_i}\phi_i(V_i(u,s))dm(s) < +\infty,
  \end{align*}
  then~$\lim_{n\to+\infty}T_n'=+\infty$.
\end{lemma}

\begin{proof}
  Note that the condition of the lemma implies that there is a constant~$\epsilon > 0$
  such that~$\PP(\forall i\in I, T_n(i) > t_i) > \epsilon$ for
  every~$n\in\NN$, hence, from the independence, it follows that this
  event occurs infinitely often in~$n\in\NN$ almost surely.

  Therefore we have that~$T_n > \epsilon$ infinitely often almost surely,
  hence~$\lim_{n\to+\infty}T_n'=+\infty$.
\end{proof}

If we have~$\lim_{n\to+\infty}T_n'=+\infty$ almost surely, then we can define further
the following random variables.

Let~$n_0 = \inf\{n\in\NN : T_n' < +\infty\}$ and~$T_{-1}=0$. Furthermore
let~$X_{T_n'}'(i) = X_n(i)$, for every~$i\in I$ and every~$n\in\NN$ with~$n<n_0$;
and~$X_t(i) = 0$ for every~$t\in\RR_+\setminus\{T_n' : n\in\NN\}$ and~$i\in I$.

Finally, let~$U_t'(i) = V_i(U_n(i),t-T_{n-1}')$ for every~$t\in\RR_+$ such
that~$T_{n-1}'\leq t<T_n'$ with~$n\in\NN$ and~$n\leq n_0$.

With these definitions (and under these hypothesis), we have stochastic
processes~$(X_t')_{t\in\RR_+}$ and~$(U_t')_{t\in\RR_+}$ on continuous time.

We now present a sufficient condition for the finitude of~$n_0$, which represents the
existence of a last neuron discharge (i.e., the event~$\sup\{t\in\RR_+ : \exists i\in
I,X_t'(i)=1\} < +\infty$).

\begin{lemma}\label{lem:n0finite}
  If~$C>0$ is such that, for every~$u\in\RR_+$ and every~$i\in I$, we have
  \begin{align*}
    \int_0^{+\infty}\phi_i(V_i(u,s))dm(s) & < C,
  \end{align*}
  then~$n_0 < +\infty$ almost surely.
\end{lemma}

\begin{proof}
  It follows directly from the fact that
  \begin{align*}
    \PP(\exists n\in\NN, T_n = +\infty)
    & = \PP(\exists n\in\NN,\forall i\in I, T_n(i) = +\infty)
    \\
    & \geq 1 - \lim_{n\to\infty}(1-e^{-C})^n = 1.
  \end{align*}
\end{proof}

\begin{remark*}
  Note that the condition of Lemma~\ref{lem:n0finite} implies the condition of
  Lemma~\ref{lem:Tnptoinfinity}.
\end{remark*}

We now come round to the theorem that characterizes the system's death in the case of
non-negative influences.

\begin{theorem}
  Suppose~$W_{i\to j}\geq 0$ for every~$i,j\in I$, let~$D$ be the digraph over~$I$
  such that~$A(D) = \{ij : i\neq j, W_{i\to j}>0\}$ and let~$S\subset I$ be the set of
  neurons~$i\in I$ such that for every~$u\in\RR_+$, we have
  \begin{align*}
    \int_0^{+\infty}\phi_i(V_i(u,t))dm(t) & < +\infty,
  \end{align*}
  and let~$R = I\setminus S$.

  Suppose also that for every~$i\in I$, we have~$a(i) > 0$ (i.e., the initial
  potentials are positive) and~$\phi_i(u) = 0$ if and only if~$u = 0$.

  Suppose finally that for every~$i\in I$, we have~$\lim_{t\to+\infty}V_i(u,t) = 0$
  if and only if~$u = 0$ or~$\lim_{t\to+\infty}V_i(u',t) = 0$ for every~$u'>u$.

  Under these circumstances, we have that
  \begin{align*}
    \PP(n_0 < +\infty) > 0
  \end{align*}
  if and only if~$D[R]$ is a~DAG (i.e., the digraph induced by the neurons
  in~$R$ is a directed acyclic graph).

  Furthermore, if~$\PP(n_0 < +\infty) > 0$, then~$\PP(n_0 < +\infty) = 1$.
\end{theorem}

\begin{proof}
  Let's first prove that, for every~$i\in R$, we have in fact
  \begin{align*}
    \int_0^{+\infty}\phi_i(V_i(u,t))dm(t) & = +\infty,
  \end{align*}
  for every~$u > 0$ (not only for some~$u$).

  We know that for such~$i\in R$, there exists~$u_0\in\RR_+$ such that
  \begin{align*}
    \int_0^{+\infty}\phi_i(V_i(u_0,t))dm(t) & = +\infty,
  \end{align*}
  which immediately gives the same for every~$u > u_0$ since~$V_i$ is non-decreasing
  on the first coordinate and~$\phi_i$ is non-decreasing.

  Suppose now that~$u\in\RR_+$ is such that~$0 < u < u_0$ and
  \begin{align*}
    \int_0^{+\infty}\phi_i(V_i(u,t))dm(t) & < +\infty.
  \end{align*}

  Since~$\phi_i$ only vanishes at~$0$, the finitude of this integral
  implies~$\lim_{t\to+\infty}V_i(u,t) = 0$ and since~$u > 0$, we have
  that~$\lim_{t\to+\infty}V_i(u_0,t) = 0$.

  This in particular means that there exists~$t_0>0$ such that~$V_i(u_0,t_0) \leq u$,
  hence, for every~$t\in\RR_+$, we have
  \begin{align*}
    V_i(u,t) \geq V_i(V_i(u_0,t_0),t) = V_i(u_0,t_0+t),
  \end{align*}
  where the inequality follows from the fact that~$V_i$ is non-decreasing in the
  first coordinate.

  Therefore we have
  \begin{align*}
    \int_0^{+\infty}\phi_i(V_i(u,t))dm(t) & \geq \int_0^{+\infty}\phi_i(V_i(u_0,t_0+t))dm(t)
    \\
    & = \int_{t_0}^{+\infty}\phi_i(V_i(u_0,t))dm(t)
    \\
    & = \int_{0}^{+\infty}\phi_i(V_i(u_0,t))dm(t) - \int_0^{t_0}\phi_i(V_i(u_0,t))dm(t)
    \\
    & = +\infty,
  \end{align*}
  where the last equality follows from the fact that the second integral on the
  left hand side is finite (because~$\phi_i$ is non-decreasing, hence locally
  bounded). But this contradicts the choice of~$u$.

  Therefore, for every~$i\in R$ and every~$u > 0$, we have
  \begin{align*}
    \int_0^{+\infty}\phi_i(V_i(u,t))dm(t) & = +\infty.
  \end{align*}

  Note finally that this implies that, for every~$i\in R$ and every~$t,u > 0$, we
  have~$V_i(u,t) > 0$ (because~$V_i$ is non-increasing on the second coordinate
  and~$\phi_i$ only vanishes at~$0$).

  \bigskip

  Suppose now that~$\PP(n_0 < +\infty) > 0$, and let's prove that~$D[R]$ is a~DAG.

  Suppose not, i.e., suppose that there are neurons~$i_1,i_2,\ldots,i_k,i_{k+1}\in R$
  with~$i_1=i_{k+1}$ and such that~$i_ji_{j+1}\in E(D)$ for every~$j\in[k]$.

  First, let's prove by induction that for every~$n\in\NN$ there exists~$j\in[k]$ such
  that~$U_n(i_j)> 0$ almost surely.

  For~$n=0$, this follows immediately from the fact that~$a(i) > 0$ for every~$i\in
  I$.

  Suppose then that~$n > 0$ and that~$U_{n-1}(i_j)>0$.

  If~$X_n(i_j)=0$, then we are done, because~$U_n(i_j)\geq V_{i_j}(U_{n-1}(i_j),T_n)
  > 0$ since all neuron influences are non-negative.
  
  Suppose then that~$X_n(i_j)=1$, then we have~$X_n(i_{j+1})=0$ almost surely,
  hence~$U_n(i_{j+1}) \geq W_{i_j\to i_{j+1}} > 0$.

  Therefore we have
  \begin{align*}
    \PP(\forall n\in\NN, \exists j\in[k], U_n(i_j) > 0) = 1.
  \end{align*}

  \medskip  

  Now, let~$E$ denote the event~$\{n_0 < +\infty\}$ and let~$A$ denote the
  event~$\{n_0 < +\infty; \exists i\in R, U_{n_0}(i) > 0\}$. Note that,
  since~$i_1,i_2,\ldots,i_k\in R$, we know that~$\PP(A\cond E) = 1$.

  On the other hand, for every~$i\in R$ and every~$n\in\NN$, we have
  \begin{align*}
    \PP(T_n(i) < +\infty \cond U_n(i))
    & = 1 - \exp\(-\int_0^{+\infty}\phi_i(V_i(U_n(i),t))dm(t)\)
  \end{align*}
  hence~$\PP(T_n(i) < +\infty \cond U_n(i) > 0) = 1$, because, for every~$u > 0$,
  we have
  \begin{align*}
    \int_0^{+\infty}\phi_i(V_i(u,t))dm(t) = +\infty.
  \end{align*}

  Finally, we have
  \begin{align*}
    0 & = \PP(\exists i\in R, T_{n_0}(i) < +\infty \cond E)
    \\
    & \geq \PP(\exists i\in R, T_{n_0}(i) < +\infty \cond A\cap E)\PP(A \cond E)
    \\
    & = \PP(A \cond E),
  \end{align*}
  which is a contradiction.

  Therefore~$D[R]$ is a~DAG.

  \bigskip

  Suppose now that~$D[R]$ is a~DAG and let's prove that~$\PP(n_0 < +\infty) = 1$
  (note that we will already prove the final part of the theorem also).

  Suppose not, i.e., suppose the event~$n_0 = +\infty$ happens with positive
  probability. From now on, all calculations and statements will be conditioned on
  the event~$\{n_0 = +\infty\}$ and on the event that two neurons never fire at the
  same time (and this will be ommited from the notation).

  The idea is to prove first that there is a state of low potentials that is visited
  infinitely often. The second step is to prove that we see infinitely often a large
  sequence of discharges only from neurons of~$R$ after reaching a state of low
  potential. Finally, the third step is to prove that there cannot be such a large
  sequence of discharges only from neurons of~$R$, which will be a contradiction.

  Before we start, let~$N = |I|$ and~$W = \sum_{i,j\in I}W_{i\to j}$.

  \medskip

  \emph{First step.}

  For every~$j\in[N]$ and every~$n\in\NN$ with~$n\geq j-1$,
  let~$A_{j,n} = \{i\in I : U_{n+1}(i) \leq jW\}$ and~$E_{j,n}$ denote the event
  \begin{align*}
    \{\lv A_{j,n}\rv \geq j\}.
  \end{align*}

  Let's prove by induction on~$j$ that~$E_{j,n}$ happens infinitely often in~$n$
  almost surely.

  For~$j = 1$, note that, since~$n_0 = +\infty$, for every~$\widetilde{n}\in\NN$, there
  exists~$n\geq\widetilde{n}$ and~$i\in I$ such that~$X_n(i) = 1$, hence~$U_{n+1}(i) =
  0$. Therefore~$E_{1,n}$ happens infinitely often in~$n$ almost surely.

  Suppose now that~$j\in[N]\setminus\{1\}$ and that~$E_{j-1,n}$ happens infinitely
  often in~$n$ almost surely.

  Suppose that~$E_{j,n}$ does not happen infinitely often in~$n$, then
  we must have that~$E_{j-1,n}\setminus E_{j,n}$ happens infinitely often in~$n$. 

  On the other hand, note that if there exists~$i_0\in I\setminus A_{j-1,n}$ such
  that~$T_n(i_0) < 1$ and~$T_n(i) > 1$ for every~$i\in A_{j-1,n}$, then we
  have~$U_{n+1}(i) \leq (j-1)W + W = jW$ for every~$i\in A_{j-1,n}\cup\{i_0\}$. This
  means that, for every~$n\in\NN$, we have
  \begin{align*}
    & \hphantom{{}\geq{}}\PP(E_{j,n+1}\cond E_{j-1,n}\setminus E_{j,n})
    \\
    & \geq \PP(\exists i_0\in I\setminus A_{j-1,n}, T_n(i_0) < 1;
    \forall i\in A_{j-1,n}, T_n(i) > 1 \cond E_{j-1,n}\setminus E_{j,n})
    \\
    & \geq \(1 - 
    \!\!\!\!\!\max_{i\in I\setminus A_{j-1,n}}\!\!\!\!\!\!
    \exp\(-\int_0^1\phi_i(V_i((j-1)W,t))dm(t)\)\)
    \!\!\!\prod_{i\in A_{j-1,n}}\!\!\!\!\!
    \exp\(-
    \!\!\int_0^1\!\!\!
    \phi_i(V_i((j-1)W,t))dm(t)\)
    \\
    & \geq \(1 - \max_{i\in I}
    \exp\(-\int_0^1\phi_i(V_i((j-1)W,t))dm(t)\)\)
    \prod_{i\in I}\exp\(-\int_0^1\phi_i(V_i((j-1)W,t))dm(t)\).
  \end{align*}

  Note that the last number is in~$(0,1)$ and is independent of~$n$. Let~$C$ be this
  number. Since~$E_{j-1,n}\setminus E_{j,n}$ happens infinitely often in~$n$, we have
  \begin{align*}
    \PP(E_{j,n} \text{ infinitely often in }n) & \geq
    1 - \lim_{n\to\infty}(1-C)^n = 1,
  \end{align*}
  which is a contradiction.

  Therefore, for every~$j\in[N]$, we have that~$E_{j,n}$ happens infinitely often in~$n$
  almost surely.

  \medskip

  \emph{Second step.}

  Let~$M = 2^{|R|}+1$ and for every~$n\in\NN$, let~$B_n = \{i\in I : X_n(i) = 1\}$
  (remember that we are conditioning on the event~$\forall n\in\NN, |B_n| = 1$).

  Moreover, for every~$j\in[M]$ and every~$n\in\NN$ with~$n\geq N+j-1$,
  let~$F_{j,n}$ denote the event
  \begin{align*}
    E_{N,n-j+1}\cap\{\forall m \in [j-1], B_{n-m+1}\subset R\},
  \end{align*}
  and note that~$F_{j+1,n+1} = F_{j,n}\cap\{B_{n+1}\subset R\}$.

  Let's prove by induction on~$j$ that~$F_{j,n}$ happens infinitely often in~$n$
  almost surely.

  For~$j=1$, we have~$F_{j,n}=E_{N,n}$ and we already know that~$E_{N,n}$ happens
  infinitely often in~$n$ almost surely.

  Suppose now that~$j\in[M]\setminus\{1\}$ and that~$F_{j-1,n}$ happens infinitely
  often in~$n$ almost surely.

  Note first that, from the definition of~$F_{j-1,n}$, we have
  \begin{align*}
    \PP(\forall i\in I, U_{n+1}(i)\leq (N+j)W \cond F_{j-1,n}) & = 1.
  \end{align*}

  Note that, for every~$n\in\NN$, we have
  \begin{align*}
    \PP(\forall n\in S, T_n(i) = +\infty \cond F_{j-1,n})) & \geq
    \prod_{i\in S}\(1 - \exp\(-\int_0^{+\infty}\phi_i(V_i((N+j)W,t))dm(t)\)\).
  \end{align*}

  Note also that right hand side does not depend on~$n$ and is a number in~$(0,1)$
  (from the definition of~$S$). Let~$K$ be this number. Since~$F_{j-1,n}$ happens
  infinitely often in~$n$, we have
  \begin{align*}
    \PP(F_{j-1,n}\cap\{\forall n\in S, T_n(i) = +\infty\})
    \text{ infinitely often in }n) & \geq
    1 - \lim_{n\to\infty}(1-K)^n = 1.
  \end{align*}

  Note now that
  \begin{align*}
    & \hphantom{{}={}}
    \PP(T_{n+1} < +\infty\cond F_{j-1,n}\cap\{\forall n\in S, T_n(i) = +\infty\})
    \\
    & = \PP(\exists i\in R, T_{n+1}(i) < +\infty
    \cond F_{j-1,n}\cap\{\forall n\in S, T_n(i) = +\infty\})
    \\
    & = \PP(F_{j,n}\cond F_{j-1,n}\cap\{\forall n\in S, T_n(i) = +\infty\}),
  \end{align*}
  and since~$n_0=+\infty$, we have that~$F_{j,n}$ happens infinitely often in~$n$
  almost surely.

  Therefore, for every~$j\in[M]$, we have that~$F_{j,n}$ happens infinitely often in~$n$
  almost surely.

  \medskip

  \emph{Third step.}

  Now, let~$i_1,i_2,\ldots,i_k$ be a topological ordering of the vertices of~$D[R]$, i.e.,
  be such that~$i_ji_l\in E(D)$ implies~$j<l$ (such an ordering always exists in a~DAG
  and can be obtained, for instance, by repeatedly removing one vertex that has
  indegree~$0$).

  Now let~$\prec$ be the strict lexicographic order induced by this order on the power
  set~$\cP(R)$ of~$R$ and, for every~$n\in\NN$, let~$Q_n = \{i\in R : U_n(i) = 0\}$.

  Note that, for every~$n\in\NN$, we have
  \begin{align*}
    \PP(Q_{n+1} \prec Q_n \cond B_n\subset R) & = 1,
  \end{align*}
  because if~$B_n = \{i\} \subset R$, then~$i\in Q_{n+1}\setminus Q_n$ (because~$i$
  must have a positive potential to fire) and the discharge of~$i$ only affects
  potentials of neurons after~$i$ in the topological ordering of~$D[R]$.

  On the other hand, since~$|\cP(R)| = 2^R = M-1$, we know that there cannot be~$M$
  consecutive occurrences of~$B_n\subset R$, because each occurrence take~$Q_n$ to a
  strictly smaller~$Q_{n+1}$. But this is precisely the definition of~$F_{M,n}$,
  which we proved to happen infinitely often in~$n$, so we have a contradiction and
  the proof is complete.
\end{proof}

\begin{observation*}
  Note that the condition~$a(i) > 0$ for every~$i\in I$ is only important for the
  first part of the proof (which is expected, since a zero initial potential should
  work in the direction of yielding~$n_0<+\infty$).

  On the other hand, the condition that for every~$i\in I$, we
  have~$\lim_{t\to+\infty}V_i(u,t) = 0$ if and only if~$u = 0$
  or~$\lim_{t\to+\infty}V_i(u',t) = 0$ for every~$u'>u$ might seem artificial at
  first, but it prevents potential decays that present distint regimens: one that
  decays to zero and others that don't.

  One example of symptomatic potential decay is
  \begin{align*}
    V_i(u,t) & = (u - \floor{u})e^{-t} + \floor{u},
  \end{align*}
  which goes to zero as~$t$ goes to~$+\infty$ if~$u < 1$, but presents a different
  behaviour for~$u \geq 1$.
\end{observation*}

\bibliographystyle{amsplain}
\bibliography{refs}
\end{document}

%% file: packages.tex
\usepackage[utf8]{inputenc} 
\usepackage[T1]{fontenc} 
\usepackage{ae} 
\usepackage{aecompl} 
\usepackage[brazil,english]{babel} 
\usepackage[fixlanguage]{babelbib} 

\usepackage{amsmath} 
\usepackage{amsthm} 
\usepackage{amsfonts} 
\usepackage{amssymb} 
\usepackage{stmaryrd} 
\usepackage{mathtools} 
\usepackage{bm} 
\usepackage{bbm} 

\usepackage{paralist} 
\usepackage{array} 

\usepackage{color} 
\usepackage[usenames,svgnames,dvipsnames]{xcolor} 

\usepackage[in]{fullpage} 
\usepackage{setspace} 
\usepackage{indentfirst} 

\usepackage[%
  pdftex,
  plainpages=false,
  pdfpagelabels,
  pagebackref,
  colorlinks,
  citecolor=black,
  linkcolor=black,
  urlcolor=black,
  filecolor=black,
  bookmarksopen
]{hyperref} 
\usepackage[all]{hypcap} 

\usepackage{boththeorems} 


%% file: common_math.tex
\newcommand{\tendsto}{\mathop{\longrightarrow}\limits}

\def\({\left(}
\def\){\right)}
\def\[{\left[}
\def\]{\right]}
\def\<{\left\langle}
\def\>{\right\rangle}
\def\lv{\left\lvert}

\def\rv{\right\rvert}

\newcommand{\floor}[1]{\ensuremath{\left\lfloor#1\right\rfloor}}

\let\geq\geqslant
\let\leq\leqslant
\let\:\colon

\let\epsilon\varepsilon

\let\phi\varphi

\def\NN{\ensuremath{\mathbb{N}}}

\def\PP{\ensuremath{\mathbb{P}}}

\def\RR{\ensuremath{\mathbb{R}}}

\def\ZZ{\ensuremath{\mathbb{Z}}}

\def\One{\ensuremath{\mathbbm{1}}}

\def\cC{\ensuremath{\mathcal{C}}}

\def\cF{\ensuremath{\mathcal{F}}}

\def\cP{\ensuremath{\mathcal{P}}}

\newenvironment{functiondef}{%
\begin{tabular}{%
    *{2}{>{$\displaystyle}r<{$}}%
    >{$\displaystyle}c<{$}%
    >{$\displaystyle}l<{$}%
}}{%
  \end{tabular}%
}


%% file: theorem_names_en.tex
\newtheoremstyle{plain}
  {\medskipamount}
  {\smallskipamount}
  {\normalfont}
  {0pt}
  {\bfseries}
  {.}
  { }
  {\thmname{#1}\thmnumber{ #2}{\normalfont\thmnote{ (#3)}}}

\theoremstyle{plain}

\newboththeorems{theorem}{Theorem}[section]
\newboththeorems{lemma}[theorem]{Lemma}
\newboththeorems{proposition}[theorem]{Proposition}
\newboththeorems{corollary}[theorem]{Corollary}
\newboththeorems{afirmation}[theorem]{Afirmation}
\newboththeorems{fact}[theorem]{Fact}
\newboththeorems{conjecture}[theorem]{Conjecture}
\newboththeorems{problem}[theorem]{Problem}
\newboththeorems{question}[theorem]{Question}
\newboththeorems{observation}[theorem]{Observation}
\newboththeorems{remark}[theorem]{Remark}

{\end{proof}\endgroup}

{\end{asparaenum}\endgroup}


%% file: mathdefs.tex
\newcommand{\comp}{\circ}
\newcommand{\cond}{\mathbin{\vert}}
